\documentclass{amsart}
\usepackage{latexsym,amsfonts,amssymb,amsmath,euscript}
\usepackage[latin1]{inputenc}
\usepackage{hyperref}
\usepackage{color,fancybox}
\newcommand{\R}{\mathbb{R}}

\numberwithin{equation}{section}
\newtheorem{theorem}{Theorem}[section]
\newtheorem{lemma}{Lemma}[section]
\newtheorem{proposition}{Proposition}[section]
\newtheorem{definition}{Definition}[section]

\newtheorem{corollary}{Corollary}[section]

\title[Perturbation of a nonautonomous problem in $\R^n$]{Perturbation of a nonautonomous problem in $\R^n$}

\date{\today}

\author[E. Capelato]{Erika Capelato}

\address[E. Capelato]{Faculdade de Ciências e Letras, UNESP - Univ Estadual Paulista, Departamento de Economia, 14800-901 Araraquara SP, Brazil}

\email{erikacap@yahoo.com.br}

\author[K. Schiabel-Silva]{Karina Schiabel-Silva}

\address[K. Schiabel-Silva]{Departamento de Matem\'atica,  Universidade Federal de S\~ao Carlos,  13565-905 S\~ao Carlos SP, Brazil}
\email{schiabel@dm.ufscar.br}

\author[R. P. Silva]{Ricardo P. Silva}

\address[R. P. Silva]{Instituto de Geoci\^{e}ncias e Ci\^{e}ncias Exatas, UNESP - Univ Estadual Paulista, Departamento de Matemática, 13506-900 Rio Claro SP, Brazil}
\email{rpsilva@rc.unesp.br}

\begin{document}

\maketitle

\begin{abstract}
In this paper we prove a stability result about the asymptotic dynamics of a perturbed nonautonomous evolution equation in $\R^n$ governed by a maximal monotone operator.
\end{abstract}

$\quad$

{\footnotesize   }

\section{Introduction}\label{sec:introd}

The investigation of the asymptotic behavior of nonlinear equations with some dissipation property and subjected to perturbations on parameters has been matter of extensive studies of several different frameworks. The goal is to understand how the variation of parameters in the models can determine the evolution of their states. For a recent works in this topic we refer the reader to  \cite{Silva:11a, Carbone:11, Silva:11,Carabalho:10, Simsen:09a} and references given there.

In this paper we analyze, from the point of view of pullback attractors theory \cite{TLR, ChV}, the asymptotic behavior of the nonlinear nonautonomous problem 
\begin{equation}\label{eq:p-lap}
\begin{gathered}
u_{t} - {\rm div}(|\nabla u|^{p-2} \nabla u) + a_\epsilon( x ) | u |^{p-2} u = B(t,u) \quad   \text{ in }   \R^n \\
u(\tau)= u_\tau  \in L^2(\R^n),
\end{gathered}
\end{equation}
with $2 < p < n$, and we also derive some stability properties with respect to small variations of the functions $a_\epsilon \in C(\R^n, \R)$. We assume that $a_\epsilon(x) \geq 1$, $\epsilon \in [0,1]$, and  $ \| a_\epsilon - a_0 \|_{L^\infty(\R^n)} \stackrel{\epsilon \to 0}{\longrightarrow} 0$. In addition, we require that $a_0$ satisfy
\begin{equation}\label{eq:a0}
\int_{\R^n} \frac{1}{a_0 (x)^{\frac{2}{p-2}}}  dx < \infty.
\end{equation}
Notice that $a_\epsilon$ also satisfies \eqref{eq:a0} for $\epsilon \in [0, \epsilon_0]$, for some $\epsilon_0 >0$.

Concerning to the nonlinearity, we suppose that $B: \R \times L^2(\R^n)\to L^2(\R^n)$ satisfies both conditions below:%
\begin{enumerate}
\item[{\rm (i)}] there exists a constant $L > 0$ such that
$$
\| B(t,u_1) - B(t,u_2) \|_{L^2(\R^n)} \leq L \|u_1 - u_2 \|_{L^2(\R^n)}, \quad  \forall \; t \in  \R, \forall \;  u_1, u_2 \in L^2(\R^n);
$$
\item[{\rm (ii)}]  the map $L_1:\R \ni t \mapsto \|B(t,0) \|_{L^2(\R^n)} \in \R$ is nondecreasing, absolutely continuous and bounded on compact subsets of $\R$. 
\end{enumerate}

Under these assumptions we state the main result of this paper:

\begin{theorem}
For each value of the parameter $\epsilon \in [0, \epsilon_0]$, the equation \eqref{eq:p-lap} generates a nonlinear compact process, $\{S_\epsilon(t,\tau): t\geq \tau \in \R \}$, in the space $L^2(\R^n)$, which has a family of pullback attractors $\{{\mathcal A_\epsilon}(t):t\in \mathbb{R}\}$. Moreover, this family of pullback attractors is upper-semicontinuous in $\epsilon = 0$.
\end{theorem}

\section{Functional Framework}\label{sec:prel}

One of the difficulties on the analysis of the asymptotic behavior of PDE's in unbounded domains, is  the lack of compactness of embeddings of some functional spaces. To overcome this obstacle, some authors have been introduced weighted Sobolev spaces, see for instance \cite{Alves:96}. In this work will use the family of auxiliary spaces
$$
E_\epsilon = \{ u \in W^{1,p}(\R^n): \displaystyle \int_{\R^n} a_\epsilon  |u|^p \, dx  < \infty \}.
$$

Next lemma is a parameter dependent adaptation of a similar result in \cite{Alves:96, Simsen:09} which we prove here for reader's convenience.

\begin{lemma}\label{lemma:E}
The space $E_\epsilon $ endowed with norm $\| u \|_{E_\epsilon } = \left[\displaystyle \int_{\R^n} \left( |\nabla u|^p + a_\epsilon |u|^p  \right) \, dx \right]^\frac{1}{p} $ is a reflexive Banach space. Furthermore $E_\epsilon \stackrel{d}{\hookrightarrow} L^r(\R^n)$, $2 \leq r \leq p^*:= \frac{p n}{n-p}$, and $E_\epsilon \subset \subset L^r(\R^n)$, $2 \leq r < p^*$, with all embedding constants independent of $\epsilon \in [0, \epsilon_0]$.
\end{lemma}
\begin{proof}
Notice that $E_\epsilon$ is a reflexive Banach space by Eberlein-Smulian theorem. 

Now if $\theta= \frac{p}{2}$, then $\frac{\theta'}{\theta}= \frac{2}{p-2}$, where $\frac{1}{\theta} + \frac{1}{\theta'}=1$. Let be $u \in E_\epsilon$, by H\"older's Inequality, 
\begin{align*}
\| u \|_{L^2(\R^n)}^2 & \leq \left[ \int_{\R^n} \frac{1}{a_\epsilon (x)^{\frac{2}{p-2}}}  dx  \right]^\frac{1}{\theta'}  \left[ \int_{\R^n} a_\epsilon(x) |u(x)|^p dx  \right]^\frac{2}{p} \\
& \leq  \left[ \int_{\R^n} \frac{1}{a_0 (x)^{\frac{2}{p-2}}}  dx + 1 \right]^\frac{1}{\theta'}  \left[ \int_{\R^n} a_\epsilon(x) |u(x)|^p dx  \right]^\frac{2}{p} \leq c \| u \|_{E_\epsilon}^2.
\end{align*}
Furthermore, recalling that $a_\epsilon(x) \geq 1$, we also have that $E_\epsilon \hookrightarrow W^{1,p}(\R^n)$ with  embedding constant independent of $\epsilon$. Therefore the embedding (part of the statement) follows from Sobolev's embedding $ W^{1,p}(\R^n) \hookrightarrow L^r(\R^n)$, $p \leq r \leq p^*$, and interpolation's inequality.

In order to prove the compactness part let us to consider a sequence $u_k \rightharpoonup 0$ in $E_\epsilon$. Since $W^{1,p}(B(0,R))\subset \subset L^2(B(0,R))$, for any $R>0$, taking subsequences if necessary, we can assume that $u_k \to 0$ in $L^2(B(0,R))$. By \eqref{eq:a0}, for each $\delta >0$, there exists $R=R(\delta)>0$ such that 
$$ 
\int_{\R^n\setminus B(0,R)} \frac{1}{a_\epsilon (x)^{\frac{2}{p-2}}}  dx < \delta,
$$
and we have
$$
\| u_k \|_{L^2(\R^n\setminus B(0,R))}^2 \leq \left[ \int_{\R^n\setminus B(0,R)} \frac{1}{a_\epsilon (x)^{\frac{2}{p-2}}}  dx  \right]^\frac{1}{\theta'}  \left[ \int_{\R^n\setminus B(0,R)} a_\epsilon(x) | u_k(x) |^p dx  \right]^\frac{2}{p}  \leq \delta^\frac{1}{\theta'}   \limsup_{k \to \infty} \| u_k\|_{E_\epsilon}^2 .
$$
Therefore 
$$
\| u_k \|_{L^2(\R^n)} = \| u_k \|_{L^2(B(0,R))} + \| u_k \|_{L^2(\R^n\setminus B(0,R))} \stackrel{k \to \infty}{\longrightarrow} 0.
$$ 

To conclude, we recall that $\{ u_k \} \subset E_\epsilon \hookrightarrow L^{p^*}(\R^n) $ is a bounded sequence, thus for all $2 < r <p^*$, one has by interpolation's inequality that
$$
\| u_k \|_{L^r(\R^n)} \leq \| u_k \|_{L^2(\R^n)}^\alpha   \| u_k \|_{L^{p^*}(\R^n)}^{1-\alpha} \to 0 .
$$
\end{proof}

\subsection{Monotone operator}

In order to rewrite the problem \eqref{eq:p-lap} in an abstract setting we consider the (nonlinear) operator $A_\epsilon : E_\epsilon \to E_\epsilon^*$ defined by
$$
\langle A_\epsilon u, v \rangle_{E^*_\epsilon,E_\epsilon} = \int_{\R^n} \left(  |\nabla u|^{p-2} \nabla u \cdot \nabla v  + a_\epsilon( x ) | u |^{p-2} u v \right) dx, \quad \forall \, v \in E_\epsilon,
$$
where $\langle \cdot, \cdot \rangle_{E^*_\epsilon,E_\epsilon} $ denotes the pair of duality between $E^*_\epsilon$ and $E_\epsilon$.

By Tartar's inequality, \cite{Vrabie:87}, one can show that the operators $A_\epsilon$, $\epsilon \in [0,\epsilon_0]$, are monotone, hemicontinuous and coercive. Let us now to consider the $L^r$-realization, $2\leq r \leq p^*$, of the operator $A_\epsilon$, denoted by $A_{\epsilon,r}$, given by
\begin{eqnarray*}
D(A_{\epsilon,r}) & = &  \{ u \in E_\epsilon :A_{\epsilon} u \in L^r(\R^n)  \}, \\ 
A_{\epsilon,r} u & = & A_{\epsilon} u, \quad \forall \, u \in D(A_{\epsilon,r})
\end{eqnarray*}

The operators $A_{\epsilon,r}$ can also be seen as the subdifferential of the lower semicontinuous convex functions $\varphi_{\epsilon,r}: L^r(\R^n) \to (-\infty, \infty]$ defined by
\begin{equation}
\varphi_{\epsilon,r}(u)= 
\begin{cases}
\frac1p \, \| u \|_{E_\epsilon}^p, & \text{ if } u \in E_\epsilon \\
\infty, & \text{otherwise.}
\end{cases}
\end{equation}
For our purposes it is of special interest the case $r=2$, and for shorten notation, we drop the index $r$ and we write $A_\epsilon$ for this realization. In this setting the problem \eqref{eq:p-lap} can be written as a quasi-linear evolution equation

\begin{equation}\label{eq:ev-abst}
\begin{array}{rcl}
u_t^\epsilon + A_\epsilon u^\epsilon  & = & B(t,u^\epsilon), \\
u^\epsilon(\tau) & = & u^\epsilon_\tau \in L^2(\R^n)
\end{array}
\end{equation}

\begin{proposition}[\cite{Brezis:73}, Proposition 3.13]\label{teo:Brezis}
Under hypothesis ${(\rm i)}$ and ${(\rm ii)}$ on $B$, for all $u^\epsilon_\tau \in L^2(\R^n)$ there exist a unique solution $u^\epsilon =u^\epsilon(\cdot, u^\epsilon_\tau) \in W^{1,1}(\tau,\infty; L^2(\R^n))$ of \eqref{eq:ev-abst}.
\end{proposition}

\section{Uniform Dissipativness}\label{sec:uniform}

In this section we establish some uniform bounds on solutions of the problem \eqref{eq:ev-abst} in order to derive existence of pullback attractors.

\begin{lemma}\label{lemma:l2bound} 
Let $u^\epsilon(\cdot, u^\epsilon_\tau) \in C([\tau, \infty), L^2(\R^n))$ be the global solution of \eqref{eq:ev-abst}. Then there exist constant $T_1$ (not dependent on $\epsilon$) and a nondecreasing function $\beta_1 : \R \to \R $, such that
$$
\| u^\epsilon(t, u^\epsilon_\tau) \|_{L^2 (\R^n)} \leq \beta_1(t), \quad \forall \; t \geq T_1 + \tau .
$$
\end{lemma}
\begin{proof}
Multiplying \eqref{eq:ev-abst} by $u^\epsilon$ and integrating over $\R^n$ we have that
\begin{equation}\label{eq:intbypart}
\frac{1}{2} \frac{d}{dt} \| u^\epsilon \|^2_{L^2(\R^n)}  + \langle A_\epsilon u^\epsilon, u^\epsilon \rangle = \langle B (t,u^\epsilon) - B(t,0), u^\epsilon \rangle + \langle B(t,0), u^\epsilon \rangle,
\end{equation}
where $\langle \cdot, \cdot \rangle$ denote the inner product in $L^2(\R^n)$. Thus
\begin{equation*}
\frac{1}{2} \frac{d}{dt} \| u^\epsilon \|^2_{L^2(\R^n)}  +  c \| u^\epsilon \|^p_{L^2(\R^n)} \leq L \| u^\epsilon \|^2_{L^2(\R^n)} +  L_1(t) \| u^\epsilon \|_{L^2(\R^n)}, 
\end{equation*}
where $c > 0$ is the constant given in Lemma \ref{lemma:E}.

Taking $\theta = \frac{p}{2}$, it follows from Young's inequality that for all $\eta > 0$,
\begin{eqnarray*}
\frac{1}{2} \frac{d}{dt} \| u^\epsilon \|^2_{L^2(\R^n)}  +  c \| u^\epsilon \|^p_{L^2(\R^n)} \leq \frac{1}{\theta'}\left( \frac{L}{\eta} \right)^{\theta'} +  \frac{\eta^\theta}{\theta} \| u^\epsilon \|^p_{L^2(\R^n)} + \frac{1}{p'}\left( \frac{L_1 (t)}{\eta} \right)^{p'}  +  \frac{\eta^p}{p} \| u^\epsilon \|^p_{L^2(\R^n)} .
\end{eqnarray*}

Choosing $\eta > 0$ such that $\gamma = c - (\frac{\eta^\theta}{\theta} + \frac{\eta^p}{p}) > 0$, it follows from \cite[Lemma 5.1]{Teman:88} that
\begin{equation}
\frac{1}{2} \| u^\epsilon \|^2_{L^2(\R^n)} \leq \left( \frac{ \delta(t)}{\gamma} \right)^{\frac2p} + \left[ \frac{\gamma}{2}(p-2)(t-\tau)\right]^{\frac{-2}{p-2}},
\end{equation}
where $ \delta(t) = \frac{1}{\theta'}\left( \frac{L}{\eta} \right)^{\theta'} + \frac{1}{p'}\left( \frac{L_1 (t)}{\eta} \right)^{p'} $. Taking $T_1 > 0$ satisfying $ \left[ \frac{\gamma}{2}(p-2)T_1\right]^{\frac{-2}{p-2}} \leq 1$,
we have for $ t \geq T_1 + \tau$ that
$$
\frac{1}{2} \| u^\epsilon \|^2_{L^2(\R^n)} \leq\left( \frac{ \delta(t)}{\gamma} \right)^{\frac2p}  +1 : = \beta_1 (t).
$$
\end{proof}

\begin{lemma}\label{lemma:Ebound}
Let $u^\epsilon(\cdot, u^\epsilon_\tau) \in C([\tau, \infty), L^2(\R^n))$ be the global solution of \eqref{eq:ev-abst}. Then there exist constant $T_2$ (not dependent on $\epsilon$) and a nondecreasing function $\beta_2 : \R \to \R $, such that
$$
\| u^\epsilon(t, u^\epsilon_\tau) \|_{E_\epsilon} \leq \beta_2(t), \quad \forall \; t \geq T_2 + \tau .
$$
\end{lemma}
\begin{proof}
By multiplying \eqref{eq:ev-abst} by $u_t^\epsilon$ we have for Young's inequality that
$$
\frac12 \| u_t^\epsilon \|^2_{L^2(\R^n)} + \frac{1}{p} \frac{d}{dt} \| u^\epsilon \|_{E_\epsilon}^p  \leq \frac{1}{2} \left( L \| u^\epsilon \|_{L^2(\R^n)} + L_1(t) \right)^2,
$$
and consequently, for $\theta= \frac{p}{2}$, we obtain
\begin{equation}\label{eq:limbound} 
\frac{1}{p} \frac{d}{dt} \| u^\epsilon \|_{E_\epsilon}^p \leq L^2 \| u^\epsilon \|_{L^2(\R^n)}^2 + L_1(t)^2 \leq \frac{1}{\theta'} L^{2 \theta'} + \frac{ 1 }{\theta} \| u^\epsilon \|_{E_\epsilon}^p + L_1(t)^2  .
\end{equation}

Fix $R > 0$ and consider the real functions $a_1=a_1(t)$ and $a_2=a_2(t)$ given by: 
$\displaystyle a_1:= \int_t^{t + R} \frac{ p }{\theta} \,  ds = \frac{ R p }{\theta} $ and $\displaystyle a_2: = \frac{R p}{\theta'} L^{2 \theta'} + R p L_1(t + R)^2 \geq \int_t^{t + R} \left( \frac{p}{\theta'} L^{2 \theta'} + p L_1(s)^2  \right) \, ds $.

Recalling \eqref{eq:intbypart}, we have by integrating from $t$ to $t +R$ that
$$
\int_t^{t+R} \| u^\epsilon \|^p_{E_\epsilon} \leq \frac{1}{2} \| u^\epsilon \|^2_{L^2(\R^n)} + L \int_t^{t+R} \| u^\epsilon (s) \|^2_{L^2(\R^n)} \, ds + L_1(t+R) \int_t^{t+R}  \| u^\epsilon (s) \|_{L^2(\R^n)} \, ds. 
$$

It follows from Lemma \ref{lemma:l2bound}, for $t \geq T_1 + \tau$ 
$$
\int_t^{t+R} \| u^\epsilon \|^p_{E_\epsilon} \leq \frac{1}{2} \beta_1(t) + R L \beta_1(t+ R) + R L_1(t+R)  \beta_1(t+ R)^{\frac12} := a_3(t).
$$

By \cite[Lemma 1.1]{Teman:88}, we obtain
\begin{equation}
\| u^\epsilon (t+ R) \|^p_{E_\epsilon} \leq \left( \frac{a_3(t)}{R} + a_2(t) \right) \, e^{a_1}:= \beta_2 (t) , \quad t \geq \tau .
\end{equation}

Choosing $R = T_1$ we have for $t-\tau \geq T_2 := 2 T_1$ that
$$
\| u^\epsilon (t) \|^p_{E_\epsilon} \leq \beta_2 (t) .
$$
\end{proof}

\section{Existence of pullback attractors}

In this section we get existence of a family $\{{\mathcal A}_\epsilon(t):t\in \mathbb{R}\}$ of pullback attractors for the problem \eqref{eq:ev-abst} as well its upper-semicontinuity in $\epsilon=0$. 

We start remembering the definition of Hausdorff semi-distance between two subsets $A$ and $B$ of a metric space $(X,d)$:
\[
\operatorname{dist}_H(A,B) = \sup_{a\in A} \inf_{b\in B} d(a,b).
\]

\begin{definition} \label{pull attraction} 
Let $\{S(t,\tau):t\geqslant \tau\in {\mathbb R}\}$ be an evolution
process in a metric space $X$. Given  $\mathcal{A}$ and $B$ subsets of $X$, we
say that $\mathcal{A}$ \emph{pullback attracts} $B$ at time $t$ if
$$
\lim_{\tau \to -\infty} \operatorname{dist}_H(S(t,\tau)B, \mathcal{A})= 0.
$$
\end{definition}

\begin{definition} \label{def4.6} 
We say that a family of subsets $\{\mathcal{A}(t):t\in \mathbb{R}\}$ of $X$ is
\emph{invariant} relatively to the evolution process
$\{S(t,\tau):t\geqslant \tau\in \mathbb{R}\}$ if $S(t,\tau) \mathcal{A}(\tau)
= \mathcal{A}(t)$,
for any $t\geqslant \tau$.
\end{definition}

\begin{definition} \label{pull-attractor}
A family of subsets $\{{\mathcal A}(t):t\in \mathbb{R}\}$ of $X$ is
called a \emph{pullback attractor} for the evolution process
$\{S(t,\tau): t\geqslant \tau\in {\mathbb R}\}$ if it is invariant,
$\mathcal{A}(t)$ is compact for all $t\in \mathbb{R}$, and pullback attracts
bounded subsets of $X$ at time $t$, for each $t\in \mathbb{R}$.
\end{definition}

The next result guarantees the existence of pullback attractors.

\begin{theorem}[\cite{Carabalho:10}]\label{teo:carabalho}
Let $\{S(t,\tau):t\geqslant \tau\in {\mathbb R}\}$ be an evolution process in a complete metric space $X$. The statements are equivalents:
\begin{enumerate}
\item[{\rm (i)}] There exist a family of compact subsets of $X$, $\{ K(t) \}_{t \in \R}$, that pullback attracts bounded sets of $X$ at time $t$;

\item[{\rm(ii)}] The process $\{S(t,\tau):t\geqslant \tau\in {\mathbb R}\}$ has a pullback attractor. 

%
\end{enumerate}
\end{theorem}

\begin{corollary}
For each value of the parameter $\epsilon \in [0, \epsilon_0]$, the equation \eqref{eq:p-lap} generates a nonlinear compact process, $\{S_\epsilon(t,\tau): t\geq \tau \in \R \}$, in the space $L^2(\R^n)$, which has a family of pullback attractors $\{{\mathcal A_\epsilon}(t):t\in \mathbb{R}\}$. 
\end{corollary}
\begin{proof}
For each value of the parameter $\epsilon \in [0,\epsilon_0]$, Proposition \ref{teo:Brezis}  guarantees that \eqref{eq:ev-abst} generates a (nonlinear) evolution process, $\{S_\epsilon(t,\tau):t\geqslant \tau\in {\mathbb R}\}$, in the space $L^2(\R^n)$, defined by $S_\epsilon(t,\tau) u^\epsilon_\tau = u^\epsilon(t,u^\epsilon_\tau)$. Lemma \ref{lemma:Ebound} shows that the family of compact sets $K_\epsilon(t) = \overline{B^{E_\epsilon}(0,\beta_2(t))}^{L^2(\R^n)}$ pullback attracts bounded sets of $L^2(\R^n)$ at time $t$. Thus by Theorem \ref{teo:carabalho} there exists a family $\{ \mathcal{A}_\epsilon(t): t \in \R \}$ of pullback attractors for $\{S_\epsilon(t,\tau):t\geqslant \tau\in {\mathbb R}\}$.
\end{proof}

\subsection{Upper-semicontinuity of pullback attractors}

Now we prove that the family of pullback attractors $\{\mathcal{A}_\epsilon(t): t \in \R \}$ is upper-semicontinuous in $\epsilon=0$, ie, we prove that
$$
\lim_{\epsilon \to 0} \operatorname{dist}_{H}(\mathcal{A}_\epsilon(t),
\mathcal{A}_0(t))=0.
$$

First, let us to consider $w^\epsilon(\cdot) = u^\epsilon(\cdot, u^\epsilon_\tau) - u^0(\cdot, u^0_\tau)$. Thus $w^\epsilon_t  +  A_\epsilon u^\epsilon - A_0 u^0 = B(t,u^\epsilon) - B(t,u^0)$. Since $a_\epsilon \geq 1$, it follows from Tartar's inequality the existence of $\alpha >0$ such that
\begin{align*}
\langle A_\epsilon u^\epsilon - A_0 u^0, w^\epsilon \rangle & =  \langle  |\nabla u^\epsilon|^{p-2} \nabla u^\epsilon - |\nabla u^0|^{p-2} \nabla u^0, w^\epsilon \rangle + \langle a_\epsilon |u^\epsilon |^{p-2}u^\epsilon - a_0 |u^0|^{p-2} u^0, w^\epsilon  \rangle \\
& =  \langle  |\nabla u^\epsilon|^{p-2} \nabla u^\epsilon - |\nabla u^0|^{p-2} \nabla u^0, w^\epsilon \rangle + \langle a_\epsilon (|u^\epsilon |^{p-2}u^\epsilon - |u^0 |^{p-2}u^0) + (a_\epsilon- a_0) |u^0|^{p-2} u^0, w^\epsilon  \rangle \\
& \geq \alpha ( \| \nabla w^\epsilon \|_{L^2(\R^n)}^p + \| w^\epsilon \|_{L^2(\R^n)}^p ) + \int_{\R^n} (a_\epsilon- a_0) |u^0|^{p-2} u^0 w^\epsilon \, dx .
\end{align*}

Therefore by Hölder's inequality 

\begin{align*}
\frac12 \frac{d}{dt} \| w^\epsilon \|^2_{L^2(\R^n)}  & \leq  -   \int_{\R^n} (a_\epsilon- a_0) |u^0|^{p-2} u^0 w^\epsilon \, dx + \| B(t,u^\epsilon) - B(t,u^0) \|_{L^2(\R^n)} \| w^\epsilon \|_{L^2(\R^n)} \\
& \leq \| a_\epsilon- a_0 \|_{L^\infty (\R^n) } \int_{\R^n} \left( | u^0 |^p + | u^0 |^{p-1}  |u^\epsilon |  \right) dx  + L \| w^\epsilon \|_{L^2 (\R^n)}^2 \\
& \leq \| a_\epsilon- a_0 \|_{L^\infty (\R^n) }  \left( \| u^0 \|_{L^p (\R^n)}^p + \| u^0 \|_{L^p (\R^n)}^{p-1}  \| u^\epsilon \|_{L^p (\R^n)}  \right) + L \| w^\epsilon \|_{L^2 (\R^n)}^2
\end{align*}

The uniform estimates given in Lemma \ref{lemma:Ebound} lead to
$$
\frac12 \frac{d}{dt} \| w^\epsilon \|^2_{L^2(\R^n)} \leq M \| a_\epsilon- a_0 \|_{L^\infty (\R^n) } + L \| w^\epsilon \|_{L^2 (\R^n)}^2,
$$
in compact subsets of $\mathbb{R}$. Integrating this last inequality from $\tau$ to $t$, we obtain
$$
\| w^\epsilon \|^2_{L^2(\R^n)} \leqslant \| u^\epsilon_\tau - u^0_\tau \|_{L^2 (\R^n)}^2 + 2M(t-\tau) \| a_\epsilon- a_0 \|_{L^\infty (\R^n) } + 2L \int_\tau^t \| w^\epsilon (s) \|^2_{L^2(\R^n)} \, ds .
$$

Hence, by Gronwall's Inequality

\begin{equation}\label{eq:continuity} 
\| w^\epsilon \|^2_{L^2(\R^n)} \leqslant \tilde{M}\left( \| u^\epsilon_\tau - u^0_\tau \|_{L^2 (\R^n)}^2 + \| a_\epsilon- a_0 \|_{L^\infty (\R^n) } \right),
\end{equation}
in compact subsets of $\mathbb{R}$. 

We can derive from this discussion the following Lemma

\begin{lemma}
Let $\{S_\epsilon(t,\tau):t\geqslant \tau\in {\mathbb R}\}$ be the process generated by the problem \eqref{eq:ev-abst}. If $u^\epsilon_\tau \stackrel{\epsilon \to 0}{\longrightarrow} u^0_\tau$ in $L^2(\R^n)$ then $S_\epsilon(t,\tau) u^\epsilon_\tau \stackrel{\epsilon \to 0}{\longrightarrow} S_0(t,\tau) u^0_\tau$ in $L^2(\R^n)$ uniformly for $t$ in compact subsets of $\R$.
\end{lemma}

\begin{corollary}
The family of pullback attractors $\{{\mathcal A_\epsilon}(t):t\in \mathbb{R}\}$ is upper-semicontinuous in $\epsilon = 0$.
\end{corollary}
\begin{proof}
Given $\delta >0$ let $\tau \in \mathbb{R}$ be such that
$\operatorname{dist}(S_0(t,\tau)B,\mathcal{A}_0(t)) < \frac{\delta}{2}$,
where $ B \supset \bigcup_{s \in \mathbb{R}}\mathcal{A}_\epsilon(s)$
is a bounded set (whose existence is
guaranteed by Lemma \ref {lemma:Ebound}).

Now for \eqref{eq:continuity}, there exists $\epsilon_0 >0$ such that
\[
 \sup_{\xi_\epsilon \in \mathcal{A}_\epsilon(t)}
\|S_\epsilon(t,\tau)\xi_\epsilon - S_0(t,\tau)\xi_\epsilon \|
<  \frac{\delta}{2},
\]
for all $\epsilon < \epsilon_0$. Then
\begin{align*}
\operatorname{dist}(\mathcal{A}_\epsilon(t),\mathcal{A}_0(t)) 
& \leqslant \operatorname{dist}(S_\epsilon(t,\tau)
 \mathcal{A}_\epsilon(\tau),S_0(t,\tau)\mathcal{A}_\epsilon(\tau))
 +\operatorname{dist}(S_0(t,\tau) \mathcal{A}_\epsilon(\tau),
 S_0(t,\tau)\mathcal{A}_0(\tau)) \\
&  = \sup_{\xi_\epsilon \in \mathcal{A}_{\epsilon}(\tau)}
 \operatorname{dist}(S_\epsilon(t,\tau)\xi_\epsilon,
 S_0(t,\tau)\xi_\epsilon)
 + \operatorname{dist}(S_0(t,\tau) \mathcal{A}_\epsilon(t),
 \mathcal{A}_0(t)) <  \frac{\delta}{2} +  \frac{\delta}{2},
  \end{align*}
which proves the upper-semicontinuity of the family of attractors.
\end{proof}

\subsection*{Acknowledgments}
The second author was partially supported by FAPESP 2008/53094-4. The third author was partially supported by FAPESP 2008/53094-4 and PROPe\textbackslash UNESP, Brazil.


\end{document}